\documentclass[12pt]{article}
\usepackage[margin=1in]{geometry}
\usepackage{amsthm, amsmath,amsfonts,amssymb,euscript,hyperref,graphics,color,slashed,mathrsfs}
\usepackage{graphicx}
\usepackage{comment}
\usepackage{import}
\usepackage{tikz}
\usepackage{latexsym}
\usepackage{mathtools}
\usepackage{appendix}
\usepackage{stmaryrd}
\usetikzlibrary{arrows}
\usepackage{pgfplots}
\usepackage{extarrows}

\newtheorem{theorem}{Theorem}[section]
\newtheorem*{theorem*}{Theorem}
\newtheorem{lemma}[theorem]{Lemma}

\newtheorem{definition}[theorem]{Definition}
\newtheorem{remark}[theorem]{Remark}

\numberwithin{equation}{section}

\begin{document}
\title {A note on the number of irrational odd zeta values, \uppercase\expandafter{\romannumeral2}}

\author{Li Lai}
\date{}

\maketitle

\begin{abstract}
We prove that there are at least $1.284 \cdot \sqrt{s/\log s}$ irrational numbers among $\zeta(3)$, $\zeta(5)$, $\zeta(7)$, $\ldots$, $\zeta(s-1)$ for any sufficiently large even integer $s$. This result improves upon the previous finding by a constant factor. The proof combines the elimination technique of Fischler-Sprang-Zudilin (2019) with the $\Phi_n$ factor method of Zudilin (2001). 
\end{abstract}

\section{Introduction}

Let $\zeta(s) := \sum_{m=1}^{+\infty} m^{-s}$~(for $\operatorname{Re}(s) > 1$) be the Riemann zeta function. We are interested in the special values $\zeta(s)$ for odd integers $s>1$ (referred to as
\emph{odd zeta values} for brevity). The aim of this note is to prove the following:

\begin{theorem}\label{main_thm}
For any sufficiently large positive integer $s$, we have
\[  \# \left\{  \text{odd~} i \in [3,s]  \mid  \zeta(i) \not\in \mathbb{Q} \right\}  \geqslant 1.284579 \cdot \sqrt{\frac{s}{\log s}}.  \]
\end{theorem}

This is a small progress on the irrationality of odd zeta values. Let us briefly review a few previously known results on this topic. In 1979, Ap\'ery \cite{Ape1979} proved that $\zeta(3)$ is irrational. The next significant development was made by Rivoal \cite{Riv2000} in 2000 (see also Ball-Rivoal \cite{BR2001}), who showed that the dimension of the $\mathbb{Q}$-linear span of $1,\zeta(3),\zeta(5),\ldots,\zeta(s)$ is at least $(1-o(1))(\log s) / (1+\log 2)$, as the odd integer $s \to +\infty$. So far, the irrationality of $\zeta(5)$ remains open, although Zudilin \cite{Zud2001} in 2001 proved that at least one of $\zeta(5),\zeta(7),\zeta(9),\zeta(11)$ is irrational. A key ingredient of Zudilin's theorem is the consideration of certain arithmetic factors, usually denoted by $\Phi_n$. We refer to it as ``$\Phi_n$ factor method'' in this note.

During 2018--2020, after a series of developments in \cite{Zud2018,Spr2018+,FSZ2019,LY2020}, it was proved that the number of irrationals among $\zeta(3),\zeta(5),\ldots,\zeta(s)$ is at least $1.192507 \sqrt{s/\log s}$ for any sufficiently large odd integer $s$ (see \cite[Thm. 1.1]{LY2020}). A key ingredient of this result is an elimination technique, whose power was first revealed by Fischler-Sprang-Zudilin in \cite{FSZ2019}. Roughly speaking, the elimination technique of Fischler-Sprang-Zudilin is useful to remove some unwanted terms in the linear combinations of $1$ and odd zeta values. 

Recently, the author \cite{Lai2024+} incorporated Zudilin's $\Phi_n$ factor method into the original proof of the Ball-Rivoal theorem \cite{BR2001} and improved the constant $1/(1+\log 2)$ to $1.108/(1+\log 2)$. But this is subsumed in a result of Fischler \cite{Fis2021+} in 2021+, where it was proved that the dimension of the $\mathbb{Q}$-linear span of $1,\zeta(3),\zeta(5),\ldots,\zeta(s)$ is at least $0.21 \sqrt{s/\log s}$ for any sufficiently large odd integer $s$. In this note, we combine the elimination technique of Fischler-Sprang-Zudilin with the $\Phi_n$ factor method of Zudilin to prove Theorem \ref{main_thm}. 

\bigskip

The structure of this note is as follows. In \S \ref{sec_2} we construct rational functions and linear forms. In \S \ref{sec_3} we study the arithmetic properties of the linear forms. In \S \ref{sec_4} we estimate the linear forms. In \S \ref{sec_5}, we use the elimination technique to transform our task into a computational problem. Finally, in \S \ref{sec_6}, we present the computational results and provide a proof for Theorem \ref{main_thm}.

\section{Rational functions and linear forms}\label{sec_2}

\subsection{Parameters}

We fix a positive integer $M$ and a finite collection of non-negative integers $\{ \delta_{j} \}_{j=1}^{J}$ such that
\[ 0 \leqslant \delta_1 \leqslant \delta_2 \leqslant \cdots \leqslant \delta_J < \frac{M}{2}. \]
Here $J \geqslant 1$ is the number of $\delta_j$'s.

Fix a positive rational number $r$. Denote by $\operatorname{den}(r)$ the denominator of $r$ in its reduced form. Eventually, we will take $r$ to be sufficiently close to a real number $r_0$ determined by the parameters $M,\delta_1,\ldots,\delta_J$.

Let $B$ be a positive real number. Let $s$ be an integral multiple of $2J$. In the sequel, we always assume that: 
\begin{align}
\text{Both $s$ and $B$ are sufficiently large, and~} s \geqslant 10(2r+1)MB^2. \label{cond_s_B}
\end{align}
Eventually, we will take $B = c\sqrt{s/\log s}$ for some constant $c>0$ depending only on the parameters $M,\delta_1,\ldots,\delta_J$. Thus, the condition $s \geqslant 10(2r+1)MB^2$ is automatically satisfied when $s$ is sufficiently large.

\subsection{Normalization factors}

As usual, we denote the Euler totient function by $\varphi(\cdot)$. 

\begin{definition}
We define the following two sets depending only on $B$:
\begin{align*}
	\Psi_B &:= \left\{ b \in \mathbb{N} \mid \varphi(b) \leqslant B \right\}, \\
	\mathcal{Z}_B &:= \left\{  \frac{a}{b} \in \mathbb{Q} \cap(0,1] \mid b \in \Psi_B, 1 \leqslant a \leqslant b \text{~and~} \gcd(a,b)=1  \right\}.
\end{align*}
Moreover, we define the following two numbers $A_1(B)$ and $A_2(B)$ depending only on $M$, $r$ and $B$:
\begin{align*}
	A_1(B) &:= \prod_{b \in \Psi_B} b^{(2r+1)M\varphi(b)}, \\
	A_2(B) &:= \prod_{b \in \Psi_B} \prod_{ p \mid b \atop p \text{~prime}} p^{(2r+1)M\varphi(b)/(p-1)}.
\end{align*}
\end{definition}

\medskip

By \cite[Prop. 2.2]{LY2020}, we have the following properties for the sets $\Psi_B$ and $\mathcal{Z}_B$:
\begin{align}
	& \left|\Psi_B\right|=\left(\frac{\zeta(2) \zeta(3)}{\zeta(6)}+o_{B \rightarrow+\infty}(1)\right) B, \label{est_Psi_B}\\
	& \left\{ \frac{1}{b}, \frac{2}{b}, \ldots, \frac{b}{b} \right\} \subset \mathcal{Z}_{B} \text{~for any~} b \in \Psi_B, \label{prop_Psi_B} \\
	& |\mathcal{Z}_B| \leqslant B^2 \text{~for any sufficiently large~} B. \label{est_Z_B}
\end{align}

By \cite[Lem. 2.4]{LY2020}, we have the following estimates for $A_1(B)$ and $A_2(B)$:
\begin{align}
	& A_1(B)=\exp \left(\left(\frac{1}{2} \frac{\zeta(2) \zeta(3)}{\zeta(6)}+o_{B \rightarrow+\infty}(1)\right)(2r+1)M B^2 \log B\right), \label{est_A_1}\\
	& A_2(B) \leqslant \exp \left(10(2 r+1)M B^2(\log \log B)^2\right) ~\text{for any sufficiently large~} B. \label{est_A_2}
\end{align}

\subsection{Rational functions}

Define the integer 
\begin{equation*}
	P_{B,\operatorname{den}(r)} := 2\operatorname{den}(r) \cdot \operatorname{LCM}_{b \in \Psi_B \atop \text{prime~} p \mid b} \left\{ p-1 \right\}, 
\end{equation*}
where $\operatorname{LCM}$ means taking the least common multiple. Note that for any $n \in P_{B,\operatorname{den}(r)}\mathbb{N}$, we have $A_1(B)^n \in \mathbb{N}$ and $A_2(B)^n \in \mathbb{N}$. For any positive integer $k$, we denote by 
\[ (x)_k := x(x+1)\cdots(x+k-1) \]
the Pochhammer symbol.

\begin{definition}\label{def_R_n}
For any $n \in P_{B,\operatorname{den}(r)}\mathbb{N}$, we define the rational function $R_n(t) \in \mathbb{Q}(t)$:
\begin{align*}
	R_n(t) :=  &A_1(B)^n A_2(B)^n \cdot \frac{\prod_{j=1}^{J} \left( (M-2\delta_j)n \right)!^{s/J} }{\left( n/\operatorname{den}(r) \right)!^{\operatorname{den}(r)(2r+1)M|\mathcal{Z}_B| - \operatorname{den}(r)M }} \\
	&\times (2t+Mn) \cdot \frac{(t-rMn)_{rMn} (t+Mn+1)_{rMn} \prod_{\theta \in \mathcal{Z}_B \setminus \{ 1 \}} (t-rMn+\theta)_{(2r+1)Mn} }{\prod_{j=1}^{J} (t+\delta_j n)_{(M-2\delta_j)n + 1}^{s/J} }.
\end{align*}
\end{definition}

\medskip

Note that the degree of the denominator of $R_n(t)$ is 
\[ \frac{s}{J} \sum_{j=1}^{J} \left( (M-2\delta_j)n + 1\right) \geqslant \frac{s}{J} \sum_{j=1}^{J} (n+1) = (n+1)s.  \]
Meanwhile, the degree of the numerator of $R_n(t)$ is 
\[ 1 + ((2r+1)M|\mathcal{Z}_B| - M)n \leqslant (2r+1)M B^2 n \]
by \eqref{est_Z_B}. Under our assumption \eqref{cond_s_B}, we conclude that
\[  \deg R_n(t) \leqslant -2 \]
for any $n \in P_{B,\operatorname{den}(r)}\mathbb{N}$. Thus, the rational function $R_n(t)$ has the unique partial-fraction decomposition of the following form.

\begin{definition}
For any $n \in P_{B,\operatorname{den}(r)}\mathbb{N}$, we define $a_{n,i,k} \in \mathbb{Q}$ ($1 \leqslant i \leqslant s$, $\delta_1 n \leqslant k \leqslant (M-\delta_1)n$) as the coefficients appearing in the partial-fraction decomposition of $R_n(t)$:
\begin{equation}\label{def_a_i_k}
	R_n(t) =: \sum_{i=1}^{s} \sum_{k= \delta_1n}^{(M-\delta_1)n} \frac{a_{n,i,k}}{(t+k)^i}.
\end{equation}
\end{definition}

\subsection{Linear forms}

\begin{definition}
For any $n \in P_{B,\operatorname{den}(r)}\mathbb{N}$, for any $\theta \in \mathcal{Z}_B$, we define 
\begin{equation}\label{def_S_n_theta}
	S_{n,\theta} := \sum_{m=0}^{+\infty} R_n(m+\theta).
\end{equation}
\end{definition}

It is easy to check that the set $\mathcal{Z}_{B} \setminus \{ 1 \}$ is invariant under the transform $x \mapsto 1-x$. Since $s$ is a multiple of $2J$ and $n$ is a multiple of $P_{B,\operatorname{den}(r)}$, all of $s/J$, $rMn$ and $(2r+1)Mn$ are even integers. Thus, the rational function $R_n(t)$ has the symmetry 
\[ R_n(-t-Mn) = - R_n(t). \] 
Then, using similar arguments as in \cite[Lem. 1]{FSZ2019}, we can express $S_{n,\theta}$ as a linear form in $1$ and Hurwitz zeta values. Recall the Hurwitz zeta values are defined by
\[ \zeta(i,\alpha) := \sum_{m=0}^{+\infty} \frac{1}{(m+\alpha)^i} \]
for any $i \in \mathbb{Z}_{\geqslant 2}$ and $\alpha > 0$.

\begin{lemma}\label{lem_lin}
For any $n \in P_{B,\operatorname{den}(r)}\mathbb{N}$ and any $\theta \in \mathcal{Z}_B$, we have
\[ S_{n,\theta} = \rho_{n,0,\theta} + \sum_{3 \leqslant i \leqslant s-1 \atop i~\text{odd} } \rho_{n,i} \zeta(i,\theta), \]
where the rational coefficients are given by
\begin{align}
\rho_{n,i} &= \sum_{k=\delta_1 n}^{(M-\delta_1)n} a_{n,i,k} \quad (3 \leqslant i \leqslant s-1,~i~\text{odd}), \label{def_rho_i} \\
\rho_{n,0,\theta} &= - \sum_{k=\delta_1 n}^{(M-\delta_1)n}\sum_{\ell= 0}^{k-1}\sum_{i=1}^{s} \frac{a_{n,i,k}}{(\ell+\theta)^i}. \label{def_rho_0_theta}
\end{align}
\end{lemma}

\medskip

Note that the coefficients $\rho_{n,i}$ ($3 \leqslant i \leqslant s-1$, $i$ odd) do not depend on $\theta \in \mathcal{Z}_B$.

\section{Arithmetic properties}\label{sec_3}

As usual, we denote by
\[ D_m = \operatorname{LCM}\{1,2,3,\ldots,m\} \]
the least commom multiple of the numbers $1,2,3\ldots,m$ for any positive integer $m$. We denote by $f^{(\lambda)}(t)$ the $\lambda$-th order derivative of a function $f(t)$ for any non-negative integer $\lambda$. For a prime number $p$, we denote by $v_p(x)$ the $p$-adic order of a rational number $x$, with the convention $v_p(0)=+\infty$.

We first state two basic lemmas. 

\begin{lemma}\label{lem_G}
Let $a,b,a_0,b_0$ be integers such that $a_0 \leqslant a \leqslant b \leqslant b_0$ and $b_0 > a_0$. Consider the rational function
\[ G(t) = \frac{(b-a)!}{(t+a)_{b-a+1}}. \]
Then, we have
	\begin{equation}\label{lem_G_1}
		D_{b_0-a_0}^{\lambda} \cdot \frac{1}{\lambda!} \left( G(t)(t+k) \right)^{(\lambda)} \big|_{t=-k} \in \mathbb{Z} 
	\end{equation}
for any integer $k \in [a_0,b_0]\cap \mathbb{Z}$ and any non-negative integer $\lambda$. 
	
Moreover, for any prime number $p > \sqrt{b_0-a_0}$, any integer $k \in [a_0,b_0]\cap \mathbb{Z}$ and any non-negative integer $\lambda$, we have
\begin{equation}\label{lem_G_2}
	v_p\left( \left( G(t)(t+k) \right)^{(\lambda)} \big|_{t=-k} \right) \geqslant -\lambda + \left\lfloor \frac{b-a}{p} \right\rfloor - \left\lfloor \frac{k-a}{p} \right\rfloor - \left\lfloor \frac{b-k}{p} \right\rfloor. 
\end{equation}
\end{lemma}

\begin{proof}
See \cite[Lemmas 16,~18]{Zud2004}. We have replaced $b$ and $b_0$ in \cite[Lemma 16,~18]{Zud2004} by $b+1$ and $b_0+1$, respectively.
\end{proof}

\medskip

\begin{lemma}\label{lem_F}
Let $a,b,c,m$ be integers with $m > 0$ and $b>0$. Let 
\[ \mu_m(b) = b^{m} \prod_{p \mid b \atop p ~\text{prime}} p^{\lfloor m/(p-1) \rfloor}. \] 
Consider the polynomial
\[ F(t) = \mu_m(b) \cdot \frac{(ct+a/b)_m}{m!}. \]
Then, we have
\[D_{m}^{\lambda} \cdot \frac{1}{\lambda!}  F^{(\lambda)}(t)  \big|_{t=-k} \in \mathbb{Z}  \]
for any integer $k \in \mathbb{Z}$ and any non-negative integer $\lambda$.
\end{lemma}

\begin{proof}
This lemma is known in the literature but not stated as above. One can slightly modify the arguments in \cite[Prop. 3.2]{LY2020} to obtain a proof.
\end{proof}

\bigskip

In the rest of this section, we will first establish the arithmetic properties of the coefficients $a_{n,i,k}$ (defined by \eqref{def_a_i_k}). Following this, we will address the coefficients $\rho_{n,i}$ (defined by \eqref{def_rho_i}) and $\rho_{n,0,\theta}$ (defined by \eqref{def_rho_0_theta}).

\begin{lemma}\label{arith_a_i_k}
For any $n \in P_{B,\operatorname{den}(r)}\mathbb{N}$ with $n > s^2$, we have
\[ \Phi_n^{-s/J} D_{(M-2\delta_1)n}^{s-i} \cdot  a_{n,i,k} \in \mathbb{Z}, \quad (1 \leqslant i \leqslant s,~\delta_1 n \leqslant k \leqslant (M-\delta_1)n),   \]
where the factor $\Phi_n$ is a product of certain primes:
\begin{equation}\label{defi_Phi}
\Phi_n := \prod_{\sqrt{Mn} < p \leqslant (M-2\delta_1)n \atop p \text{~prime}} p^{\omega(n/p)}, 
\end{equation}
and the function $\omega(\cdot)$ is defined by
\begin{equation}\label{defi_omega}
\omega(x) = \min_{0 \leqslant y < 1} \sum_{j=1}^{J}  \left( \lfloor(M-2\delta_j)x\rfloor - \lfloor y-\delta_j x \rfloor - \lfloor (M-\delta_j)x - y \rfloor \right).
\end{equation}
\end{lemma}

\begin{proof}
Fix any $i \in \{1,2,\ldots,s\}$ and any $k \in [\delta_1n,~(M-\delta_1)n] \cap \mathbb{Z}$. By \eqref{def_a_i_k}, we have
\begin{equation}\label{a_i_k_determined_by_R_n}
a_{n,i,k} = \frac{1}{(s-i)!} \left( R_n(t)(t+k)^{s} \right)^{(s-i)} \Big|_{t=-k}.
\end{equation}

Define the following rational functions (we use the notation $\mu_m(b)$ in Lemma \ref{lem_F}): 
\begin{align*}
F_0(t) &:= 2t+Mn, \\
F_{1,-,\nu}(t) &:= \frac{(t-rMn+ (\nu-1)n/\operatorname{den}(r))_{n/\operatorname{den}(r)}}{\left( n/\operatorname{den}(r) \right)!}, \quad \nu=1,2,\ldots,\operatorname{den}(r) rM, \\
F_{1,+,\nu}(t) &:= \frac{(t+Mn+1 + (\nu -1)n/\operatorname{den}(r))_{n/\operatorname{den}(r)}}{\left( n/\operatorname{den}(r) \right)!}, \quad \nu=1,2,\ldots,\operatorname{den}(r) rM, \\
F_{\theta,\nu}(t) &:= \mu_{n/\operatorname{den}(r)}(\operatorname{den}(\theta)) \cdot \frac{(t - rMn + \theta + (\nu-1)n/\operatorname{den}(r))_{n/\operatorname{den}(r)}}{\left(n/\operatorname{den}(r) \right)!}, \\ 
&\qquad\qquad\qquad\qquad \theta \in \mathcal{Z}_{B} \setminus \{1\}, \quad \nu = 1,2,\ldots, \operatorname{den}(r)(2r+1)M, \\
G_j(t) &:= \frac{((M-2\delta_j)n)!}{\left( t+\delta_j n \right)_{(M-2\delta_j)n+1}}, \quad j=1,2,\ldots,J.
\end{align*}
By Definition \ref{def_R_n}, we have
\begin{align}
	R_n(t)(t+k)^s = &F_0(t) \prod_{\nu=1}^{\operatorname{den}(r) rM} (F_{1,-,\nu}(t)F_{1,+,\nu}(t)) \prod_{\theta \in \mathcal{Z}_B \setminus \{1\} }\prod_{\nu=1}^{\operatorname{den}(r)(2r+1)M} F_{\theta,\nu}(t) \notag\\
	&\times \prod_{j=1}^{J} (G_j(t)(t+k))^{s/J}. \label{R=FFFGGG}
\end{align}
By Lemma \ref{lem_F}, for any polynomial $F(t)$ of the form $F_0(t)$, $F_{1,\pm,\nu}(t)$, $F_{\theta,\nu}(t)$, we have
\begin{equation}\label{F_is_good}
D_{n/\operatorname{den}(r)}^{\lambda} \cdot \frac{1}{\lambda!} F^{(\lambda)}(t) \big|_{t=-k} \in \mathbb{Z} 
\end{equation} 
for any non-negative integer $\lambda$.
By \eqref{lem_G_1} of Lemma \ref{lem_G} (with $a_0 = \delta_1n$ and $b_0 = (M-\delta_1)n$), we have
\begin{equation}\label{G_is_good}
D_{(M-2\delta_1)n}^{\lambda} \cdot \frac{1}{\lambda!} \left( G_j(t)(t+k) \right)^{(\lambda)} \big|_{t=-k} \in \mathbb{Z} 
\end{equation}
for any $j=1,2,\ldots,J$ and any non-negative integer $\lambda$. Now, substituting \eqref{R=FFFGGG} into \eqref{a_i_k_determined_by_R_n}, applying the Leibniz rule, and using \eqref{F_is_good}\eqref{G_is_good}, we obtain
\begin{equation}\label{arith_anik_1}
D_{(M-2\delta_1)n}^{s-i} \cdot  a_{n,i,k} \in \mathbb{Z}.
\end{equation} 

Moreover, for any prime $p$ such that $\sqrt{Mn} < p \leqslant (M-2\delta_1)n$, taking also \eqref{lem_G_2} of Lemma \ref{lem_G} into consideration, and noting that $p > s$ (because $n > s^2$), we obtain
\begin{align}
v_p\left( a_{n,i,k} \right) &=  v_p\left( \left( R_n(t)(t+k)^{s} \right)^{(s-i)} \Big|_{t=-k} \right) \notag\\
	&\geqslant -(s-i) + \frac{s}{J}\cdot\sum_{j=1}^{J} \left( \left\lfloor \frac{(M-2\delta_j)n}{p}  \right\rfloor - \left\lfloor \frac{k-\delta_jn}{p}  \right\rfloor - \left\lfloor \frac{(M-\delta_j)n - k}{p}  \right\rfloor \right)  \notag\\
	&\geqslant -(s-i) + \frac{s}{J}\cdot\omega\left( \frac{n}{p} \right), \label{arith_anik_2}
\end{align}
where the function $\omega(\cdot)$ is given by \eqref{defi_omega}. Combining \eqref{arith_anik_1} and \eqref{arith_anik_2}, we obtain the desired conclusion:
\[ \Phi_n^{-s/J} D_{(M-2\delta_1)n}^{s-i} \cdot  a_{n,i,k} \in \mathbb{Z}, \quad (1 \leqslant i \leqslant s,~\delta_1n \leqslant k \leqslant (M-\delta_1)n).   \]
\end{proof}

\begin{lemma}\label{arith_rho_i}
For any $n \in P_{B,\operatorname{den}(r)}\mathbb{N}$ with $n > s^2$, we have 
\[ \Phi_n^{-s/J} D_{(M-2\delta_1)n}^{s-i} \cdot  \rho_{n,i} \in \mathbb{Z}, \quad (3 \leqslant i \leqslant s-1,~i~\text{odd}).  \]
\end{lemma}

\begin{proof}
It follows directly from \eqref{def_rho_i} and Lemma \ref{arith_a_i_k}.
\end{proof}

\begin{lemma}\label{arith_rho_0_theta}
For any $n \in P_{B,\operatorname{den}(r)}\mathbb{N}$ with $n > s^2$, for any $\theta \in \mathcal{Z}_B \setminus \{ 1 \}$, we have 
\[  \Phi_n^{-s/J} D_{(M-2\delta_1)n}^{s} \cdot  \rho_{n,0,\theta} \in \mathbb{Z}.  \]

\end{lemma}

\begin{proof}
We prove by contradiction. Let $\theta \in \mathcal{Z}_B \setminus \{1\}$ and suppose that 
\[ \Phi_n^{-s/J} D_{(M-2\delta_1)n}^{s} \cdot  \rho_{n,0,\theta} \not\in \mathbb{Z}. \]
By substituting \eqref{def_rho_0_theta} into above equation, we obtain that, there exist integers $k_0$, $\ell_0$ such that $\delta_1 n \leqslant k_0 \leqslant (M-\delta_1)n$, $0 \leqslant \ell_0 < k_0$, and
\[ \Phi_n^{-s/J} D_{(M-2\delta_1)n}^{s}  \sum_{i=1}^{s} \frac{a_{n,i,k_0}}{(\ell_0+\theta)^i} \not\in \mathbb{Z}.  \] 
Note that $1 - \theta \in \mathcal{Z}_B \setminus \{ 1 \}$ and $t+k_0-\ell_0-\theta$ is a factor of $(t-rMn+1-\theta)_{(2r+1)Mn}$. Thus, $t+k_0-\ell_0-\theta$ is a factor of the numerator of $R_n(t)$ (see Definition \ref{def_R_n}) and
\[ R_n(-k_0+\ell_0 + \theta) = 0. \]
Then, by \eqref{def_a_i_k} we have
\[ \Phi_n^{-s/J} D_{(M-2\delta_1)n}^{s}  \sum_{i=1}^{s} \frac{a_{n,i,k_0}}{(\ell_0+\theta)^i} =  - \Phi_n^{-s/J} D_{(M-2\delta_1)n}^{s}  \sum_{i=1}^{s}\sum_{k = \delta_1 n \atop k \neq k_0}^{(M-\delta_1)n} \frac{a_{n,i,k}}{(k-k_0+\ell_0+\theta)^i} \not\in \mathbb{Z}. \]
Therefore, there exist a prime $p$, two indices $i_0,i_1 \in \{1,2,\ldots,s\}$, and an index $k_1 \in \{ \delta_1 n, \ldots, (M-\delta_1)n \}$ with $k_1 \neq k_0$ such that
\[ v_p\left( \Phi_n^{-s/J} D_{(M-2\delta_1)n}^{s} \cdot \frac{a_{n,i_0,k_0}}{(\ell_0+\theta)^{i_0}} \right) < 0, \quad  v_p\left( \Phi_n^{-s/J} D_{(M-2\delta_1)n}^{s} \cdot \frac{a_{n,i_1,k_1}}{(k_1-k_0+\ell_0+\theta)^{i_1}} \right) < 0. \]
By Lemma \ref{arith_a_i_k}, we have 
\[ v_p\left( \Phi_n^{-s/J} D_{(M-2\delta_1)n}^{s-i_0} \cdot a_{n,i_0,k_0}\right) \geqslant 0, \quad v_p\left( \Phi_n^{-s/J} D_{(M-2\delta_1)n}^{s-i_1} \cdot a_{n,i_1,k_1}\right) \geqslant 0. \]
Thus,
\[ v_p\left( \ell_0 + \theta \right) > v_p(D_{(M-2\delta_1)n}), \quad v_p\left( k_1 - k_0 + \ell_0 + \theta \right) > v_p(D_{(M-2\delta_1)n}). \]
It follows that $v_p(k_1-k_0) > v_p(D_{(M-2\delta_1)n})$, which is absurd because $0 < |k_1-k_0| \leqslant (M-2\delta_1)n$. 
\end{proof}

\medskip

\begin{remark}
For $\theta = 1$, we can prove that
\[ \Phi_n^{-s/J} \prod_{j=1}^{J} D_{M_j n}^{s/J} \cdot  \rho_{n,0,1} \in \mathbb{Z}, \]
where $M_j = \max\{ M-2\delta_1, M-\delta_j \}$. The proof is the same as that of \cite[Lem. 4.4]{Lai2024+}. In some sense, the arithmetic property of $\rho_{n,0,1}$ is the worst among $\rho_{n,0,\theta}$ ($\theta \in \mathcal{Z}_B$). We will eliminate $\rho_{n,0,1}$ in our final linear forms (see the proof of Lemma \ref{eli}). 
\end{remark}

\section{Asymptotic estimates}\label{sec_4}

Recall that $S_{n,\theta}$ ($\theta \in \mathcal{Z}_B$) are defined by \eqref{def_S_n_theta}. We have the following asymptotic estimates for $S_{n,\theta}$ ($\theta \in \mathcal{Z}_B$) as $n \in P_{B,\operatorname{den}(r)}\mathbb{N}$ and $n \to +\infty$.

\begin{lemma}\label{lem_ana}
For $\theta =1 \in \mathcal{Z}_B$, we have
\[ \lim_{n \to +\infty} S_{n,1}^{1/n} = g(x_0), \]
where 
\begin{align*}
	g(x) :=& A_1(B)A_2(B) \operatorname{den}(r)^{(2r+1)M|\mathcal{Z}_B| - M} \left( \prod_{j=1}^{J} (M-2\delta_j)^{M-2\delta_j} \right)^{s/J} \\
	&\times \left( (2r+1)M + x \right)^{(2r+1)M|\mathcal{Z}_B|} \cdot \frac{(rM+x)^{rM}}{((r+1)M+x)^{(r+1)M}} \\
	&\times \prod_{j=1}^{J} \left( \frac{(rM+\delta_j+x)^{rM+\delta_j}}{\left( (r+1)M - \delta_j + x \right)^{(r+1)M-\delta_j}} \right)^{s/J},
\end{align*}
and $x_0$ is the unique positive real solution of the equation $f(x)=1$ with
\begin{align*}
	f(x) := \left( \frac{(2r+1)M+x}{x} \right)^{|\mathcal{Z}_B|} \frac{rM+x}{(r+1)M + x} \prod_{j=1}^{J} \left( \frac{rM+\delta_j+x}{(r+1)M-\delta_j+x} \right)^{s/J}.
\end{align*}
Moreover, for any $\theta \in \mathcal{Z}_B$, we have
\[ \lim_{n \to +\infty} \frac{S_{n,1}}{S_{n,\theta}}  = 1.  \]
\end{lemma}

\begin{proof}
We first prove that the equation $f(x)=1$ has a unique positive real solution $x=x_0$. A straightforward calculation gives
\[ \frac{f'(x)}{f(x)} = \frac{u(x)}{x((2r+1)M+x)},  \quad x \in (0,+\infty), \]
where the function $u(x)$ is defined on $(0,+\infty)$ by
\begin{align}
u(x) :=&  -(2r+1)M|\mathcal{Z}_B| + M\left( 1 - \frac{(rM)((r+1)M)}{(rM+x)((r+1)M+x)} \right) \notag\\
&+  \frac{s}{J}\sum_{j=1}^{J} (M-2\delta_j)\left( 1 - \frac{(rM+\delta_j)((r+1)M-\delta_j)}{(rM+\delta_j+x)((r+1)M-\delta_j+x)} \right). \label{def_u}
\end{align}
Noting that $M-2\delta_j \geqslant 1$ and 
\begin{align*}
	1 - \frac{(rM+\delta_j)((r+1)M-\delta_j)}{(rM+\delta_j+x)((r+1)M-\delta_j+x)} &= \frac{x((2r+1)M+x)}{(rM+\delta_j+x)((r+1)M-\delta_j+x)} \\
	&> \frac{x}{(2r+1)M + x}, \quad j=1,2,\ldots,J,
\end{align*}
we have
\begin{equation}\label{u>}
	u(x) >  -(2r+1)M|\mathcal{Z}_B| + s \cdot \frac{x}{(2r+1)M + x}, \quad x \in (0,+\infty). 
\end{equation}
Let 
\[ x_2 := \frac{(2r+1)^2M^2|\mathcal{Z}_B|}{s - (2r+1)M|\mathcal{Z}_B|}, \quad (\text{we have~} x_2 > 0 \text{~by \eqref{cond_s_B} and \eqref{est_Z_B}}) \]
which is the root of the right-hand side of \eqref{u>}. Then $u(x_2)>0$. On the other hand, By \eqref{def_u}, the function $u(x)$ is increasing on $(0,+\infty)$ and $u(0^{+}) = -(2r+1)M|\mathcal{Z}_B| < 0$. Thus, 
there exists a unique $x_1 \in (0,x_2)$ such that $u(x_1) = 0$. Therefore, $f(x)$ is decreasing on $(0,x_1)$ and increasing on $(x_1,+\infty)$. Since $f(0^{+}) = +\infty$ and $f(+\infty) = 1$, there exists a unique $x_0 \in (0,x_1)$ such that $f(x_0) = 1$. Moreover, $f(x) > 1$ for $x \in (0,x_0)$ and $f(x)<1$ for $x \in (x_0,+\infty)$. By \eqref{est_Z_B}, we have 
\begin{equation}\label{x_0_<}
	x_0 < x_2 = \frac{(2r+1)^2M^2|\mathcal{Z}_B|}{s - (2r+1)M|\mathcal{Z}_B|} \leqslant \frac{(2r+1)^2M^2B^2}{s - (2r+1)MB^2}.
\end{equation}

The rest of the proof is similar to that of \cite[Lem. 4.1]{LY2020}. We only give a sketch. Since $R_n(m+\theta) = 0$ for $m=0,1,\ldots,rMn - 1$ and $\theta \in \mathcal{Z}_B$, we have 
\[ S_{n,\theta} = \sum_{k=0}^{+\infty} R_n(rMn+k+\theta). \]
We can express $R_n(rMn+k+\theta)$ almostly as Gamma ratio:  
\begin{align*}
	R_n(rMn+k+\theta) =&  A_1(B)^n A_2(B)^n \cdot \frac{\prod_{j=1}^{J} \left( (M-2\delta_j)n \right)!^{s/J} }{\left( n/\operatorname{den}(r) \right)!^{\operatorname{den}(r)(2r+1)M|\mathcal{Z}_B| - \operatorname{den}(r)M }} \\
	&\times ((2r+1)Mn+2k+2\theta) \cdot \prod_{\theta' \in \mathcal{Z}_B} \frac{\Gamma((2r+1)Mn+k+\theta+\theta')}{\Gamma(k+\theta+\theta')} \\
	&\times \frac{\Gamma\left( rMn + k + \theta \right)}{\Gamma\left( (r+1)Mn + k + 1 + \theta \right)} \cdot \prod_{j=1}^{J} \left( \frac{\Gamma\left( (rM+\delta_j)n + k + \theta \right)}{\Gamma\left( ((r+1)M-\delta_j)n + k + 1 + \theta \right)} \right)^{s/J}.
\end{align*}
Suppose for the moment that $k = \kappa n$ for some constant $\kappa > 0$, then by Stirling's formula we have (as $n \to +\infty$)
\begin{align*}
	R_n(rMn+\kappa n+\theta) &= n^{O(1)} \cdot A_1(B)^n A_2(B)^n \operatorname{den}(r)^{((2r+1)M|\mathcal{Z}_B|-M)n} \left( \prod_{j=1}^{J} (M-2\delta_j)^{(M-2\delta_j)n} \right)^{s/J} \\
	&\times \left( \frac{((2r+1)M+\kappa)^{((2r+1)M+\kappa)n}}{\kappa^{\kappa n}} \right)^{|\mathcal{Z}_B|} \cdot \frac{(rM+\kappa)^{(rM+\kappa)n}}{((r+1)M + \kappa)^{((r+1)M + \kappa)n}} \\
	&\times \prod_{j=1}^{J} \left(  \frac{(rM+\delta_j+\kappa)^{(rM+\delta_j+\kappa)n}}{((r+1)M-\delta_j+\kappa)^{((r+1)M-\delta_j+\kappa)n}}  \right)^{s/J} \\
	&= n^{O(1)} \cdot \left( f(\kappa)^{\kappa}g(\kappa) \right)^n.
\end{align*}
Define the function $h(x):=f(x)^x g(x)$ on $(0,+\infty)$. Then a direct calculation gives $h'(x) / h(x) = \log f(x)$. So $h(x)$ is increasing on $(0,x_0)$ and decreasing on $(x_0,+\infty)$. Based on the above observation, one can argue as in the proof of \cite[Lem. 4.1]{LY2020} to show that $S_{n,\theta} = n^{O(1)} h(x_0)^n$ for any $\theta \in \mathcal{Z}_B$, and hence
\[ \lim_{n \to +\infty} S_{n,\theta}^{1/n} = h(x_0) = g(x_0). \]
Moreover, for any fixed small $\varepsilon_0 > 0$ and any $\theta \in \mathcal{Z}_B$, we have
\begin{align*}
S_{n, 1}&=(1+o(1)) \sum_{\left(x_0-\varepsilon_0\right) n \leqslant k \leqslant\left(x_0+\varepsilon_0\right) n} R_n\left(m_1 n+k+1\right), \\
S_{n, \theta}&=(1+o(1)) \sum_{\left(x_0-\varepsilon_0\right) n \leqslant k \leqslant\left(x_0+\varepsilon_0\right) n} R_n\left(m_1 n+k+\theta\right).
\end{align*}
Then, by using $\Gamma(x+1-\theta)/\Gamma(x) = (1+o_{x \to +\infty}(1)) x^{1-\theta}$, we obtain uniformly for $\left(x_0-\varepsilon_0\right) n \leqslant k \leqslant\left(x_0+\varepsilon_0\right) n$ that, as $n \to +\infty$,
\[ \frac{R_n\left( rMn+k+1\right)}{R_n\left( rMn+k+\theta\right)}=(1+o(1)) f(k/n)^{1-\theta}. \]
Therefore,
\[ f\left(x_0+\varepsilon_0\right)^{1-\theta} \leqslant \liminf _{n \rightarrow+\infty} \frac{S_{n, 1}}{S_{n, \theta}} \leqslant \limsup _{n \rightarrow+\infty} \frac{S_{n, 1}}{S_{n, \theta}} \leqslant f\left(x_0-\varepsilon_0\right)^{1-\theta}. \]
Letting $\varepsilon_0 \to 0^{+}$, we obtain
\[ \lim_{n \to +\infty} \frac{S_{n,1}}{S_{n,\theta}}  = 1.  \]
\end{proof}

Next, we estimate the factor $\Phi_n$ (defined by \eqref{defi_Phi}). It is well known that the prime number theorem implies 
\begin{equation}\label{PNT}
	D_m =e^{m+o(m)}, \quad \text{as~} m \to +\infty. 
\end{equation}
The following lemma is also a corollary of the prime number theorem. For details, we refer the reader to \cite[Lem. 4.4]{Zud2002}.
\begin{lemma}\label{est_Phi_n}
We have
	\[ \Phi_n = \exp\left( \varpi n + o(n) \right), \quad \text{as~} n \to +\infty,  \]
where the constant 
\begin{equation}\label{defi_varpi}
	\varpi = \int_{0}^{1} \omega(x)~\mathrm{d}\psi(x) - \int_{0}^{1/(M-2\delta_1)} \omega(x)~\frac{\mathrm{d}x}{x^2}.  
\end{equation}
The function $\psi(\cdot)$ is the digamma function, and the function $\omega(\cdot)$ is given by \eqref{defi_omega}.
\end{lemma}

\section{Elimination procedure}\label{sec_5}

We use the elimination technique of Fischler-Sprang-Zudilin \cite{FSZ2019} to prove the following. 

\begin{lemma}\label{eli}
If 
\begin{equation}\label{we_need}
\log (g(x_0)) + s \cdot \left( -\frac{\varpi}{J} + (M-2\delta_1) \right) < 0,    
\end{equation}  
then there are at least $|\Psi_B| - 1$ irrational numbers among $\zeta(3)$, $\zeta(5)$, $\ldots$, $\zeta(s-1)$.
\end{lemma}

\begin{proof}
Suppose that the number of irrationals among $\zeta(3)$, $\zeta(5)$, $\ldots$, $\zeta(s-1)$ is less than $|\Psi_B| - 1$; then we can take a subset $I \subset \{ 3,5,\ldots,s-1 \}$ with $|I|= |\Psi_B| -2$ such that $\zeta(i) \in \mathbb{Q}$ for all $i \in \{ 3,5,\ldots,s-1 \} \setminus I$. Let $A \in \mathbb{N}$ be a common denominator of these rational zeta values; that is, $A \cdot \zeta(i) \in \mathbb{Z}$ for all $i \in \{ 3,5,\ldots,s-1 \} \setminus I$.

Since the generalized Vandermonde matrix
\[  \left[  b^{i} \right]_{b \in \Psi_B,~i \in \{ 0,1 \} \cup I}  \]
is invertible (see \cite[\S 4]{FSZ2019}), there exist integers $w_b \in \mathbb{Z}$ ($b \in \Psi_B$) such that 
\begin{align*}
\sum_{b \in \Psi_B} w_b b^{i} &= 0 \quad \text{for any~} i \in \{ 0 \}  \cup I, \text{~and} \\
\sum_{b \in \Psi_B} w_b b  &\neq 0.
\end{align*}

By \eqref{prop_Psi_B}, we have $k/b \in \mathcal{Z}_B$ for any $b \in \Psi_B$ and any $k =1,2,\ldots,b$; thus $S_{n,k/b}$ is defined by Definition \ref{def_S_n_theta}. For any $n \in P_{B,\operatorname{den}(r)}\mathbb{N} $ with $n>s^2$, we define 
\[ \widetilde{S}_{n} := \sum_{b \in \Psi_B} w_b \sum_{k=1}^{b} S_{n,k/b}.  \]
By Lemma \ref{lem_lin} and the fact
\[ \sum_{k=1}^{b} \zeta(i,k/b) = b^{i} \zeta(i), \]
we have
\begin{align*}
\widetilde{S}_{n} &= \sum_{b \in \Psi_B} w_b \sum_{k=1}^{b} \rho_{n,0,k/b} + \sum_{i \in \{3,5,\ldots,s-1\}} \left( \sum_{b \in \Psi_B} w_b b^{i} \right)\rho_{n,i}\zeta(i) \\
&= \sum_{b \in \Psi_B} w_b \sum_{k=1}^{b-1} \rho_{n,0,k/b} + \sum_{i \in \{3,5,\ldots,s-1\}\setminus I} \left( \sum_{b \in \Psi_B} w_b b^{i} \right)\rho_{n,i}\zeta(i).
\end{align*}
By Lemma \ref{lem_ana}, we have 
\[ \widetilde{S}_{n} = \left( \sum_{b \in \Psi_B} w_b b + o(1) \right)S_{n,1}, \quad \text{as~} n \to +\infty. \]

By Lemma \ref{arith_rho_i} and Lemma \ref{arith_rho_0_theta}, we have
\[ A \cdot \Phi_n^{-s/J}  D_{(M-2\delta_1) n}^{s} \cdot \widetilde{S}_{n} \in \mathbb{Z}. \]
On the other hand, by Eq. \eqref{PNT}, Lemma \ref{est_Phi_n}, Lemma \ref{lem_ana} and Eq. \eqref{we_need}, we have
\[ 0 < \lim_{n \to +\infty} \left|   A \cdot \Phi_n^{-s/J} D_{(M-2\delta_1) n}^{s} \cdot \widetilde{S}_{n} \right|^{1/n} = g(x_0)\exp\left( s \cdot \left( -\frac{\varpi}{J} + (M-2\delta_1) \right) \right) < 1,  \]
a contradiction.
\end{proof}

\bigskip

The next theorem simplifies our task by transforming it into a computational problem.

\begin{theorem}\label{final_lemma}
Fix any positive integers $M$ and $J$, fix any non-negative integers $\delta_1,\delta_2,\ldots,\delta_J$ such that 
\[ 0 \leqslant \delta_1 \leqslant \delta_2 \leqslant \cdots \leqslant \delta_J < \frac{M}{2}. \]
Let $\varpi$ be the real number determined by $M,\delta_1,\delta_2,\ldots,\delta_J$ as follows:
\[ \varpi = \int_{0}^{1} \omega(x)~\mathrm{d}\psi(x) - \int_{0}^{1/(M-2\delta_1)} \omega(x)~\frac{\mathrm{d}x}{x^2},  \]
where $\psi(x) = \Gamma'(x)/\Gamma(x)$ is the digamma function, and
\[ \omega(x) = \min_{0 \leqslant y < 1} \sum_{j=1}^{J}  \left( \lfloor(M-2\delta_j)x\rfloor - \lfloor y-\delta_j x \rfloor - \lfloor (M-\delta_j)x - y \rfloor \right). \]
Let $r_0 \in (-\delta_1/M, +\infty)$ be the real number determined by $M,\delta_1,\delta_2,\ldots,\delta_J$ as follows:
\begin{align}
&\sum_{j=1}^{J} (M-2\delta_j) \Big( \log((r_0+1)M-\delta_j) + \log(r_0M+\delta_j) \Big)  \notag\\
=& -2\varpi + 2J(M-2\delta_1) + 2\sum_{j=1}^{J}  (M-2\delta_j)\log(M-2\delta_j). \label{def_r_0}
\end{align}
If $r_0 > 0$, then we have
\[ \# \left\{  i \in \{3,5,\ldots,s-1\}  \mid  \zeta(i) \not\in \mathbb{Q} \right\}  \geqslant (C_0-o(1)) \cdot \sqrt{\frac{s}{\log s}},  \]
as the even integer $s \to +\infty$, where the constant 
\[ C_0 = \sqrt{\frac{2\zeta(2)\zeta(3)}{\zeta(6)} \cdot \frac{1}{J}\sum_{j=1}^{J} \log \frac{(r_0+1)M - \delta_j}{r_0M + \delta_j}  }. \]
\end{theorem}

\begin{proof}
Let $s$ be a sufficiently large multiple of $2J$. Let $B = c \sqrt{s/\log s}$ for some constant $c>0$. Let $r \in \mathbb{Q}_{>0}$. By the definition of $g(x_0)$ (see Lemma \ref{lem_lin}) and Eqs. \eqref{est_Z_B}\eqref{est_A_1}\eqref{est_A_2}\eqref{x_0_<}, we have
\begin{align*}
&\lim_{s \to +\infty \atop B=c\sqrt{s/\log s}} \frac{\log(g(x_0))}{s} \\ 
=&~ \frac{1}{2}\frac{\zeta(2)\zeta(3)}{\zeta(6)} (2r+1)M \cdot \frac{c^2}{2} + \frac{1}{J} \sum_{j=1}^{J} (M-2\delta_j)\log(M-2\delta_j) \\
&+ \frac{1}{J} \sum_{j=1}^{J} \Big( (rM+\delta_j)\log(rM+\delta_j) -  \big((r+1)M-\delta_j \big)\log\big((r+1)M-\delta_j \big)  \Big).
\end{align*}
If the constant $c$ satisfies
\begin{equation}\label{we_need_for_c}
c^2 < \frac{4\zeta(6)}{\zeta(2)\zeta(3)} \cdot F(r), 
\end{equation}
where
\begin{align*}
F(r):= \frac{1}{(2r+1)MJ} \cdot \Bigg(  & \sum_{j=1}^{J}  \big((r+1)M-\delta_j\big)\log\big((r+1)M-\delta_j\big) \\
&- \sum_{j=1}^{J} (rM+\delta_j)\log(rM+\delta_j)  \\
&- \sum_{j=1}^{J}  (M-2\delta_j)\log(M-2\delta_j) \\
&+\varpi - J(M-2\delta_1)  \Bigg),  
\end{align*}
then \eqref{we_need} is satisfied when $s$ is sufficiently large, and hence
\begin{equation}\label{lower_bound}
\# \left\{  i \in \{3,5,\ldots,s-1\}  \mid  \zeta(i) \not\in \mathbb{Q} \right\}  \geqslant  (1+o(1)) \frac{\zeta(2)\zeta(3)}{\zeta(6)} \cdot c \cdot \sqrt{\frac{s}{\log s}}, 
\end{equation}
as $s \to +\infty$, by Lemma \ref{eli} and Eq. \eqref{est_Psi_B}.

So we want to maximize $F(r)$. Note that
\[ F'(r) = \frac{1}{(2r+1)^2 MJ} \cdot G(r),  \]
where 
\begin{align*}
	G(r) := &-2\varpi +2J(M-2\delta_1) + 2\sum_{j=1}^{J} (M-2\delta_j)\log(M-2\delta_j) \\
	&- \sum_{j=1}^{J} (M-2\delta_j)\Big( \log((r+1)M-\delta_j) + \log(rM+\delta_j) \Big).
\end{align*}
Obviously, $G(r)$ is decreasing on $(-\delta_1/M, +\infty)$. Since $G((-\delta_1/M)^+) = +\infty$ and $G(+\infty) = -\infty$, there exists a unique $r_0 \in (-\delta_1/M, +\infty)$ satisfying \eqref{def_r_0}; that is, $G(r_0)=0$. So $F(r_0)$ is the maximum value for $F(r)$. We have $F(r_0)>0$ because $F(r)$ is decreasing on $(r_0,+\infty)$ and $F(+\infty) = 0$.

Fix any small $\varepsilon > 0$. If $r_0 > 0$, then we can take $r \in \mathbb{Q}_{>0}$ sufficiently close to $r_0$, such that $F(r) > (1-\varepsilon/2) F(r_0)$. Take 
\[ c = \sqrt{ \left( 1-\frac{\varepsilon}{2} \right) \cdot \frac{4\zeta(6)}{\zeta(2)\zeta(3)} \cdot F(r_0) }, \]
then \eqref{we_need_for_c} is satisfied, and hence by \eqref{lower_bound} we have 
\begin{align*}
\# \left\{  i \in \{3,5,\ldots,s-1\}  \mid  \zeta(i) \not\in \mathbb{Q} \right\}  &\geqslant \left( 1-\frac{\varepsilon}{2} \right) \sqrt{ \frac{4\zeta(2)\zeta(3)}{\zeta(6)} F(r_0) } \cdot \sqrt{\frac{s}{\log s}}  \\
&= \left( 1-\frac{\varepsilon}{2} \right) C_0 \cdot \sqrt{\frac{s}{\log s}}
\end{align*}
for any sufficiently large $s$ which is a multiple of $2J$. Thus,
\[ \# \left\{  i \in \{3,5,\ldots,s-1\}  \mid  \zeta(i) \not\in \mathbb{Q} \right\}  \geqslant (1-\varepsilon)C_0 \cdot \sqrt{\frac{s}{\log s}} \]
for any sufficiently large even integer $s$.

\end{proof}

\section{Computations}\label{sec_6}

In this section, we will take explicit parameters $M$, $\delta_1,\delta_2,\ldots,\delta_J$ in Theorem \ref{final_lemma} and calculate the corresponding $\varpi$, $r_0$ and $C_0$. We refer the reader to \cite[\S 5]{LZ2022} for the method of calculating $\varpi$ by \texttt{MATLAB}.

If we take $M=1$, $J=1$, $\delta_1 = 0$, then $\varpi = 0$, $r_0 = 2.263884\ldots$ and $C_0 = 1.192507\ldots$. Thus, we rediscover the result in \cite{LY2020}: for any sufficiently large even integer $s$, 
\[ \# \left\{  i \in \{3,5,\ldots,s-1\}  \mid  \zeta(i) \not\in \mathbb{Q} \right\}  \geqslant 1.192507 \cdot \sqrt{\frac{s}{\log s}}. \]

The simplest parameters to improve the above result are: $M=7$, $J=2$, $\delta_1=0$, $\delta_2 = 1$. For this collection of parameters, we have $\varpi = 2.284309\ldots$, $r_0 = 1.850170\ldots$ and 
\[ C_0 = 1.197980\ldots. \]

Take $M=57$, $J=18$, $(\delta_1,\delta_2,\ldots,\delta_{18}) = (1,1,1,1,1,2,2,3,3,5,5,6,6,7,8,9,10,12)$. Then $\varpi = 238.966249\ldots$, $r_0 = 1.573948\ldots$ and
\[ C_0 = 1.262672\ldots.  \]

\bigskip

\begin{proof}[Proof of Theorem \ref{main_thm}]
Take $M=563$, $J=76$ and $\delta_1,\delta_2,\ldots,\delta_{76}$ in Table \ref{table}.
\begin{table}[htbp]
	\centering
	\caption{The choice of $(\delta_1,\ldots,\delta_{76})$}
	\label{table}
	\begin{tabular}{|cccccccc|}
		\hline
		$\delta_1=1$ & $\delta_{11}=4$ & $\delta_{21}=14$ & $\delta_{31}=24$ & $\delta_{41}=44$ & $\delta_{51}=64$ & $\delta_{61}=84$ & $\delta_{71}=104$  \\
		$\delta_2=1$ & $\delta_{12}=5$ & $\delta_{22}=15$ &  $\delta_{32}=26$ & $\delta_{42}=46$ & $\delta_{52}=66$ & $\delta_{62}=86$ & $\delta_{72}=108$ \\
		$\delta_3=1$ & $\delta_{13}=6$ & $\delta_{23}=16$ & $\delta_{33}=28$ & $\delta_{43}=48$ & $\delta_{53}=68$ & $\delta_{63}=88$ & $\delta_{73}=112$ \\
		$\delta_4=1$ & $\delta_{14}=7$ & $\delta_{24}=17$ & $\delta_{34}=30$ & $\delta_{44}=50$ & $\delta_{54}=70$ & $\delta_{64}=90$ & $\delta_{74}=116$ \\
		$\delta_5=1$ & $\delta_{15}=8$ & $\delta_{25}=18$ & $\delta_{35}=32$ & $\delta_{45}=52$ & $\delta_{55}=72$ & $\delta_{65}=92$ & $\delta_{75}=120$ \\
		$\delta_6=2$ & $\delta_{16}=9$ & $\delta_{26}=19$ & $\delta_{36}=34$ & $\delta_{46}=54$ & $\delta_{56}=74$ & $\delta_{66}=94$ & $\delta_{76}=124$ \\
		$\delta_7=2$ & $\delta_{17}=10$ & $\delta_{27}=20$ & $\delta_{37}=36$ & $\delta_{47}=56$ & $\delta_{57}=76$ & $\delta_{67}=96$ &\\
		$\delta_8=3$ & $\delta_{18}=11$ & $\delta_{28}=21$ & $\delta_{38}=38$ & $\delta_{48}=58$ & $\delta_{58}=78$ & $\delta_{68}=98$ &\\
		$\delta_9=3$ & $\delta_{19}=12$ & $\delta_{29}=22$ & $\delta_{39}=40$ & $\delta_{49}=60$ & $\delta_{59}=80$ & $\delta_{69}=100$ &\\
		$\delta_{10}=4$ & $\delta_{20}=13$ & $\delta_{30}=23$ & $\delta_{40}=42$ & $\delta_{50}=62$ & $\delta_{60}=82$ & $\delta_{70}=102$ &\\
		\hline
	\end{tabular}
\end{table}

Then $\varpi = 12694.987927\ldots$, $r_0 = 1.502726\ldots$ and
\[ C_0 = 1.284579\ldots. \]
By Theorem \ref{final_lemma}, we have
\[ \# \left\{  i \in \{3,5,\ldots,s-1\}  \mid  \zeta(i) \not\in \mathbb{Q} \right\}  \geqslant 1.284579 \cdot \sqrt{\frac{s}{\log s}} \]
for any sufficiently large even integer $s$.
\end{proof}

\begin{remark}
The same $\delta_j$'s in Table \ref{table} were used in \cite{Lai2024+}. But the parameter $M=563$ in this note is different from that in \cite{Lai2024+}. Indeed, if we fix $J=76$ and these $\delta_j$'s in Table \ref{table}, and let $M$ vary from $500$ to $600$, then $M=563$ gives the best $C_0$. 
\end{remark}

\vspace*{3mm}
\begin{flushright}
\begin{minipage}{148mm}\sc\footnotesize
L.\,L., Beijing International Center for Mathematical Research, Peking University, Beijing, China\\
{\it E--mail address}: {\tt lilaimath@gmail.com} \vspace*{3mm}
\end{minipage}
\end{flushright}


\begin{thebibliography}{99}

\bibitem{Ape1979} R. Ap{\'e}ry, \textit{Irrationalit{\'e} de $\zeta(2)$ et $\zeta(3)$}, in \textit{Journ{\'e}es Arithm{\'e}tiques (Luminy, 1978)}, Ast{\'e}risque, vol. 61 (Soci{\'e}t{\'e} Math{\'e}matique de France, Paris, 1979), 11--13.


\bibitem{BR2001} K. Ball and T. Rivoal, \textit{Irrationalité d’une infinité de valeurs de la fonction zêta aux entiers impairs}, Invent. math. 146, 193–207 (2001).


\bibitem{Fis2021+} S. Fischler, \textit{Linear independence of odd zeta values using Siegel's lemma}, arXiv:2109.10136 [math.NT], 2021.

\bibitem{FSZ2019}  S. Fischler, J. Sprang and W. Zudilin, \textit{Many odd zeta values are irrational}, Compos. Math. 155(5) (2019), 938--952.

\bibitem{Lai2024+} L. Lai, \textit{A slight improvement on the Ball-Rivoal theorem}, arXiv:2407.14236 [math.NT], 2024.

\bibitem{LY2020} L. Lai and P. Yu, \textit{A note on the number of irrational odd zeta values}, Compos. Math. 156 (2020), no. 8, 1699--1717.

\bibitem{LZ2022} L. Lai and L. Zhou, \textit{At least two of $\zeta(5),\zeta(7),\ldots,\zeta(35)$ are irrational}, Publ. Math. Debrecen 101/3--4 (2022), 353--372.


\bibitem{Riv2000} T. Rivoal, \textit{La fonction z{\^e}ta de Riemann prend une infinit{\'e} de valeurs irrationnelles aux entiers impairs},  C. R. Acad. Sci. Paris S{\'e}r. I Math. 331 (2000), no. 4, 267--270.

\bibitem{Spr2018+} J. Sprang, \textit{Infinitely many odd zeta values are irrational. By elementary means}, arXiv: 1802.09410 [math.NT], 2018.


\bibitem{Zud2001} W. Zudilin, \textit{One of the numbers $\zeta(5)$, $\zeta(7)$, $\zeta(9)$, $\zeta(11)$ is irrational}, Uspekhi Mat. Nauk [Russian Math. Surveys] 56 (2001), 149--150 [774--776].

\bibitem{Zud2002} W. Zudilin, \textit{Irrationality of values of the Riemann zeta function}, Izvestiya Ross. Akad. Nauk Ser. Mat. [Izv. Math.] 66 (2002), 49--102 [489--542].

\bibitem{Zud2004} W. Zudilin, \textit{Arithmetic of linear forms involving odd zeta values}, J. Th{\'e}or. Nombres Bordeaux 16(1), 251--291 (2004).

\bibitem{Zud2018}  W. Zudilin, \textit{One of the odd zeta values from $\zeta(5)$ to $\zeta(25)$ is irrational. By elementary means}, SIGMA Symmetry Integrability Geom. Methods Appl. 14 (2018), no. 028.
\end{thebibliography}
\end{document}